\documentclass[reqno]{amsart}
\usepackage{amsmath, amssymb, paralist, cite, color}
\usepackage[all]{xypic}
\definecolor{e-mail}{rgb}{0,.40,.80}
\definecolor{reference}{rgb}{.20,.60,.22}
\definecolor{citation}{rgb}{0,.40,.80}
\usepackage[pdfstartview=FitH, colorlinks, linkcolor=reference, citecolor=citation, urlcolor=e-mail, backref ]{hyperref}
\newtheorem{thm}{Theorem}
\newtheorem{cor}[thm]{Corollary}
\newtheorem{lem}[thm]{Lemma}

\theoremstyle{definition}
\newtheorem{defn}[thm]{Definition}
\theoremstyle{remark}
\newtheorem{rem}[thm]{Remark}
\numberwithin{thm}{section}
\theoremstyle{definition}

\theoremstyle{definition}

\theoremstyle{definition}
\numberwithin{equation}{section}

 \title[Computing the unipotent radical of a $\mathrm{PPV}$-group]{Computation of the unipotent radical of the differential Galois group for a parameterized second-order linear differential equation}
 \author{Carlos E. Arreche}
\email{carreche@gc.cuny.edu}
\address{Mathematics Department, The Graduate Center of the City University of New York, New York, NY 10016}
 \thanks{This material is based upon work partially supported by a National Science Foundation (NSF) Graduate Research Fellowship (grant 40017-04-05) and by NSF grant CCF-0952591.}

\begin{document}


\begin{abstract} We propose a new method to compute the unipotent radical $R_u(H)$ of the differential Galois group $H$ associated to a parameterized second-order homogeneous linear differential equation of the form \[\tfrac{\partial^2}{\partial x^2}Y-qY=0,\] where $q \in F(x)$ is a rational function in $x$ with coefficients in a $\Pi$-field $F$ of characteristic zero, and $\Pi$ is a commuting set of parametric derivations. The procedure developed by Dreyfus reduces the computation of $R_u(H)$ to solving a creative telescoping problem, whose effective solution requires the assumption that the maximal reductive quotient $H / R_u(H)$ is a $\Pi$-constant linear differential algebraic group. When this condition is not satisfied, we compute a new set of parametric derivations $\Pi'$ such that the associated differential Galois group $H'$ has the property that $H'/ R_u(H')$ is $\Pi'$-constant, and such that $R_u(H)$ is defined by the same differential equations as $R_u(H')$. Thus the computation of $R_u(H)$ is reduced to the effective computation of $R_u(H')$. We expect that an elaboration of this method will be successful in extending the applicability of some recent algorithms developed by Minchenko, Ovchinnikov, and Singer to compute unipotent radicals for higher order equations.\end{abstract}

\maketitle


\section{Introduction} Consider a linear differential equation of the form \begin{equation}\label{equation}\delta_x^nY+\sum_{i=0}^{n-1}r_i\delta_x^iY=0,\end{equation} where $r_i\in K:=F(x)$, the field of rational functions in $x$ with coefficients in a $\Pi$-field $F$, $\delta_x$ denotes the derivative with respect to $x$, and $\Pi:=\{\partial_1,\dots,\partial_m\}$ is a set of commuting derivations. Letting $\Delta:=\{\delta_x\}\cup\Pi$, consider $K$ as a $\Delta$-field by setting $\partial_j x=0$ for each $j$. The parameterized Picard-Vessiot theory developed in \cite{cassidy-singer:2006} associates a parameterized Picard-Vessiot ($\mathrm{PPV}$) group to such an equation. In analogy with classical Galois theory and the Picard-Vessiot theory developed by Kolchin \cite{kolchin:1948}, the $\mathrm{PPV}$-group measures the $\Pi$-algebraic relations amongst the solutions to \eqref{equation}. The differential Galois groups that arise in this theory are linear differential algebraic groups, which are the differential-algebraic analogues of linear algebraic groups: they are subgroups of $\mathrm{GL}_n$ that are defined by the vanishing of systems of polynomial differential equations in the matrix entries. The study of linear differential algebraic groups was pioneered in \cite{cassidy:1972}. The parameterized Picard-Vessiot theory of \cite{cassidy-singer:2006} is a special case of an earlier generalization of Kolchin's theory, which is developed in \cite{landesman:2008}, as well as the differential Galois theory for difference-differential equations with parameters presented in \cite{hardouin-singer:2008}. 

We propose a new method to compute the unipotent radical $R_u(H)$ of the $\mathrm{PPV}$-group $H$ corresponding to a second-order parameterized linear differential equation of the following form \begin{equation}\label{original}\delta_x^2Y-qY=0,\end{equation} where $q\in F(x)=:K$, and $F$ is a $\Pi$-field. In \cite{dreyfus:2011}, Dreyfus applies results from \cite{dreyfus-thesis} to develop algorithms to compute $H$ (see also \cite{arreche:2012} for a detailed discussion of Dreyfus' results in the setting of one parametric derivation). This algorithm is extended to arbitrary second-order linear differential equations over $K$ in \cite{arreche-gl2}. In \cite[\S2.1]{dreyfus:2011}, the computation of $R_u(H)$ is reduced to the computation of a finite set of linear differential operators in $F[\Pi]$, and a method is given for their computation which is known to halt under the assumption that the maximal reductive quotient $H/R_u(H)$ is $\Pi$\emph{-constant} (cf. \cite[Alg.~1]{minchenko-ovchinnikov-singer:2013a}).

We circumvent this obstruction by modifying the set of parametric derivations. Letting $\mathcal{D}:= F\cdot\Pi$ denote the $F$-vector space spanned by $\Pi$, we compute a commuting $F$-basis $\Pi'$ for the Lie subspace $\mathcal{L}\subseteq\mathcal{D}$ consisting of derivations $\partial'\in\mathcal{D}$ such that the matrix entries of $H/R_u(H)$ are $\partial'$-constant. We let the $\Pi'$-linear differential algebraic group $H'$ denote the $\mathrm{PPV}$-group of \eqref{unimodular}, obtained by replacing $\Pi$ with the new set of parametric derivations $\Pi'$ throughout the previous discussion. Theorem~\ref{main1} states that the maximal reductive quotient $H'/R_u(H')$ is $\Pi'$-constant, and that $R_u(H)$ is defined by the same differential equations as $R_u(H')$. Heuristically, for the purposes of computing $R_u(H)$ one may safely disregard any derivation $\partial\in\mathcal{D}$ that doesn't ``think'' that $H/R_u(H)$ is differentially constant. Thus the computation of $R_u(H)$ is reduced to the effective computation of $R_u(H')$ prescribed by Dreyfus' procedure. This has the consequence, counterintuitive in light of the fact that the known algorithms \cite{dreyfus:2011,chen-singer-etal:2014,minchenko-ovchinnikov-singer:2013a} for computing $R_u(H)$ all require that $H/R_u(H)$ be $\Pi$-constant, that the computation of $R_u(H)$ should be \emph{easier} to perform when $H/R_u(H)$ fails to be $\Pi$-constant, since in this case there are less derivations appearing in the defining equations for $R_u(H)$.

In \S\ref{examples}, we apply our main result (Theorem~\ref{main1}) to compute the $\mathrm{PPV}$-group of a concrete parameterized linear differential equation \eqref{eg1} over $K$. An application of Theorem~\ref{main1} already appears in \cite{arreche:2013}. Consider the $\mathrm{PPV}$-group $G_\gamma$ associated to the \emph{incomplete Gamma function}, which is defined by $\gamma(t,x):=\int_0^xs^{t-1}e^{-s}ds$ for $\mathrm{Re}(t)>0$, and extended analytically to a multivalued meromorphic function on $\mathbb{C}\times\mathbb{C}$. It satisfies: \begin{equation}\label{gamma}\frac{\partial^2\gamma}{\partial x^2}-p\frac{\partial\gamma}{\partial x}=0,\end{equation} where $p:=\frac{1-t-x}{x}$. This is an example of a parameterized linear differential equation over $\mathbb{C}(x,t)$, where $\delta_x:=\frac{\partial}{\partial x}$, and $\Pi:=\{\frac{\partial}{\partial t}\}$ is the parametric derivation. In \cite{johnson:1995} it was shown that $\gamma(t,x)$ does not satisfy any polynomial differential equations over $\mathbb{C}(x,t,e^x,\log (x))$ with respect to $\frac{\partial}{\partial t}$. This result was necessary in \cite[Ex.~7.2]{cassidy-singer:2006} for the computation of $G_\gamma$; specifically, to conclude that $R_u(G_\gamma)=\mathbb{G}_a$, the additive group. In \cite[Thm.~3.2]{arreche:2013}, whose proof follows the ideas of \cite[Cor.~3.4.1]{hardouin-singer:2008}, the differential transcendence of the solutions to \eqref{gamma} was characterized in terms of two conditions on the coefficient $p$. Letting $G$ (resp., $G'$) denote the $\mathrm{PPV}$-group (resp., the non-parameterized $\mathrm{PV}$-group) corresponding to \eqref{gamma} for an arbitrary $p\in \mathbb{C}(x,t)$, the first condition is equivalent to the statement that $G/R_u(G)$ is not $\frac{\partial}{\partial t}$-constant, and the second condition says that $R_u(G')=\mathbb{G}_a$. When the first condition is satisfied, it is proved in \cite[Lem. 3.6(2)]{hardouin-singer:2008} that $R_u(G)$ is either $0$ or $\mathbb{G}_a$. Since $G$ is Zariski-dense in $G'$ by \cite[Prop. 3.6(2)]{cassidy-singer:2006}, it follows that $R_u(G)=\mathbb{G}_a$ precisely when $R_u(G')=\mathbb{G}_a$. In other words, whenever $G/R_u(G)$ is not $\frac{\partial}{\partial t}$-constant, then $R_u(G)$ is defined by the same (non-differential) equations as $R_u(G')$. Since $G'$ is the ``$\mathrm{PPV}$-group'' for \eqref{gamma} obtained by replacing $\Pi=\{\frac{\partial}{\partial t}\}$ with $\Pi'=\emptyset$, Theorem~\ref{main1} includes \cite[Thm.~3.2]{arreche:2013} as a special case. In fact, the strategy for the proof of Theorem~\ref{main1} is similar to the one followed in \cite[Thm.~3.2]{arreche:2013}. But in order to carry out this strategy in the more general setting of several parametric derivations, it is necessary to strengthen some of the technical results used in the proof of \cite[Thm.~3.2]{arreche:2013}. This is done in Lemma~\ref{surj-res1} and Lemma~\ref{surj-res2} (see also Remark~\ref{surj1-rem} and Remark~\ref{surj1-rem}).

An algorithm to compute the unipotent radical $R_u(G)$ of the $\mathrm{PPV}$-group $G$ corresponding to an $n^\text{th}$-order parameterized linear differential equation over $K$ is given in \cite{minchenko-ovchinnikov-singer:2013a}, under the familiar assumption that $G/R_u(G)$ is $\Pi$-constant. We expect that an elaboration of the methods presented in this paper will be successful in extending the procedure of \cite{minchenko-ovchinnikov-singer:2013a} to compute $R_u(G)$ in cases where $G/R_u(G)$ is not necessarily $\Pi$-constant. Algorithms to compute \emph{telescopers} for rational functions, algebraic functions, and hyperexponential functions, are given in \cite{chen-singer:2012}, \cite{chen-kauers-singer:2012}, and \cite{chen-etal:2013}, respectively. The notion of \emph{parallel telescoping} investigated in \cite{chen-singer-etal:2014} leads to algorithms \cite[\S5]{chen-singer-etal:2014} to compute $\mathrm{PPV}$-groups in the setting of several principal derivations and one parametric derivation (see \S\ref{ppv-intro}). The creative telescoping problems solved by these algorithms lie at the core of the algorithms presented in \cite{dreyfus:2011} (see also \cite{arreche:2012}) to compute the $\mathrm{PPV}$-group of \eqref{original}, where the computation of $R_u(H)$ is reduced to solving a creative telescoping problem. The hyper-exponential assumption, which is necessary to apply the algorithms of \cite{chen-kauers-singer:2012,chen-singer-etal:2014}, coincides in this case with the requirement that $H/R_u(H)$ be $\Pi$-constant (see \cite[\S4]{chen-singer-etal:2014} and Remark~\ref{piprime}). We refer to \cite[\S1]{chen-singer-etal:2014} for more details and references concerning the connection between creative telescoping problems and the computation of $\mathrm{PPV}$-groups.


\section{Preliminaries}\label{preliminaries}
See \cite{kolchin:1976,vanderput-singer:2003} for more details concerning the following definitions. A $\Delta$\emph{-ring} is a ring $A$ equipped with a finite set $\Delta:=\{\delta_1,\dots,\delta_m\}$ of commuting derivations (that is, $\delta_i(ab)=a\delta_i (b)+ \delta_i( a)b$ and $\delta_i\delta_j=\delta_j\delta_i$ for each $a,b\in A$ and $1\leq i,j \leq m$). We often omit the parentheses, and write $\delta a$ for $\delta(a)$. For $\Pi\subseteq\Delta$, we denote the subring of $\Pi$\emph{-constants} of $A$ by $A^\Pi:=\{a\in K \ | \ \delta a=0, \ \delta\in\Pi\}$. When $\Pi=\{\delta\}$ is a singleton, we write $A^\delta$ instead of $A^\Pi$. If $A=K$ happens to be a field, we say that $(K,\Delta)$ is a $\Delta$-field. Every field is assumed to be of characteristic zero.

The \emph{ring of differential polynomials} over $K$ (in $m$ differential indeterminates) is denoted by $K\{Y_1,\dots,Y_m\}_\Delta$. Algebraically, it is the free $K$-algebra in the countably infinite set of variables $\{\theta Y_i \ | \ 1\leq i \leq m, \ \theta\in\Theta\}$, where $$\Theta:=\{\delta_1^{r_1}\dots\delta_n^{r_n} \ | \ r_i\in\mathbb{Z}_{\geq 0} \ \ \text{for} \ \ 1\leq i \leq n\}$$ is the free commutative monoid on the set $\Delta$. The ring $K\{Y_1,\dots,Y_m\}_\Delta$ carries a natural structure of $\Delta$-ring, given by $\delta_i(\theta Y_j):=(\delta_i\cdot\theta)Y_j$.

We say $\mathbf{p}\in K\{Y_1,\dots,Y_m\}_\Delta$ is a \emph{linear differential polynomial} if it belongs to the $K$-vector space spanned by the $\theta Y_j$, for $\theta\in\Theta$ and $1\leq j \leq m$. The $K$-vector space of linear differential polynomials will be denoted by $K\{Y_1,\dots,Y_m\}_\Delta^1$.

The \emph{ring of linear differential operators} $K[\Delta]$ is the $K$-span of $\Theta$. Its (non-commutative) ring structure is defined by composition of additive endomorphisms of $K$. The canonical identification of (left) $K$-vector spaces $K[\Delta]\simeq K\{Y\}_{\Delta}^1$ given by $\sum_\theta a_\theta\theta\leftrightarrow\sum_\theta a_\theta\theta Y$ will be assumed implicitly in what follows.

If $M$ is a $\Delta$-field and $K$ is a subfield such that $\delta(K)\subset K$ for each $\delta\in \Delta$, we say $K$ is a $\Delta$\emph{-subfield} of $M$ and $M$ is a $\Delta$\emph{-field extension} of $K$. If $y_1,\dots,y_n\in M$, we denote the $\Delta$-subfield of $M$ generated over $K$ by all the derivatives of the $y_i$ by $$K\langle y_1,\dots,y_n\rangle_\Delta\subseteq M.$$

A $\Delta$-field $K$ is $\Delta$\emph{-closed} if every system of polynomial differential equations defined over $K$ that admits a solution in some $\Delta$-field extension of $K$ already has a solution in $K$. This notion is discussed at length in \cite{kolchin:1974} (see also \cite{cassidy-singer:2006, trushin:2010}).


\subsection{Linear differential algebraic groups and parameterized Picard-Vessiot theory} \label{ppv-intro}
We recall some standard facts from the parameterized Picard-Vessiot theory \cite{cassidy-singer:2006} (see also \cite{gill-gor-ov:2012, landesman:2008,hardouin-singer:2008}) and the theory of linear differential algebraic groups \cite{cassidy:1972} (see also \cite{kolchin:1985, ovchinnikov:2009}). Let $F$ be a $\Pi$-field, where $\Pi:=\{\partial_1,\dots,\partial_m\}$, and let $K:=F(x)$ be the field of rational functions in $x$ with coefficients in $F$. Let $\Delta:=(\{\delta_x\}\cup\Pi)$, and consider $K$ as a $\Delta$-field by setting $\delta_xx=1$, $K^{\delta_x}=F$, and $\partial_ix=0$ for each $i$. We will sometimes refer to $\delta_x$ as the \emph{main} derivation, and to $\Pi$ as the set of \emph{parametric} derivations. Consider the following linear differential equation with respect to the main derivation, where $r_i\in K$ for each $0\leq i\leq n-1$:\begin{equation}\label{ppv-eq} \delta_x^nY+\sum_{i=0}^{n-1}r_i\delta_x^iY=0.\end{equation}

\begin{defn}  A $\Delta$-field extension $M\supseteq K$ is a \emph{parameterized Picard-Vessiot extension} (or $\mathrm{PPV}$-extension) of $K$ for~\eqref{ppv-eq} if:
\begin{enumerate}[(i)]
\item There exist $n$ distinct, $F$-linearly independent elements $y_1,\dots, y_n\in M$ such that $\delta_x^ny_j+\sum_ir_i\delta_x^iy_j=0$ for each $1\leq j \leq n$.
\item $M=K\langle y_1,\dots y_n\rangle_\Delta$.
\item $M^{\delta_x}=K^{\delta_x}$.
\end{enumerate}

We define the \emph{parameterized Picard-Vessiot group} (or $\mathrm{PPV}$-group) as the group of $\Delta$-automorphisms of $M$ over $K$, and we denote it by $\mathrm{Gal}_\Delta(M/K)$. The $F$-linear span of all the $y_j$ is the \emph{solution space} $\mathcal{S}$. \end{defn} 

If $F$ is $\Pi$-closed, it is shown in \cite{cassidy-singer:2006} that a $\mathrm{PPV}$-extension and $\mathrm{PPV}$-group for \eqref{ppv-eq} over $K$ exist and are unique up to $K$-$\Delta$-isomorphism. Although this assumption allows for a simpler exposition of the theory, several authors \cite{gill-gor-ov:2012, wibmer:2011} have shown that, in many cases of practical interest, the parameterized Picard-Vessiot theory can be developed without assuming that $F$ is $\Pi$-closed. In any case, we may always embed $F$ in a $\Pi$-closed field \cite{kolchin:1974,trushin:2010}.

The action of $\mathrm{Gal}_\Delta(M/K)$ is determined by its restriction to $\mathcal{S}$, which defines an embedding $\mathrm{Gal}_\Delta(M/K)\hookrightarrow\mathrm{GL}_n(F)$ after choosing an $F$-basis for $\mathcal{S}$. It is shown in \cite{cassidy-singer:2006} that this embedding identifies the $\mathrm{PPV}$-group with a linear differential algebraic group.

\begin{defn}[\protect{\cite{cassidy:1972,kolchin:1985}}] \label{ldag-def}Let $F$ be a $\Pi$-closed field. We say that a subgroup $G\subseteq \mathrm{GL}_n(F)$ is a \emph{linear differential algebraic group} if $G$ is defined as a subset of $\mathrm{GL}_n(F)$ by the vanishing of a system of polynomial differential equations in the matrix entries, with coefficients in $F$. We say that $G$ is $\Pi$\emph{-constant} if it is conjugate to a subgroup of $\mathrm{GL}_n(F^\Pi)$.\end{defn}

The study of linear differential algebraic groups was pioneered in \cite{cassidy:1972}, where the differential algebraic subgroups of the additive group $\mathbb{G}_a(F)$ and the multiplicative group $\mathbb{G}_m(F)$ were classified in terms of finite sets of linear differential operators (see \cite[Prop.~11, Prop.~31 and its Corollary]{cassidy:1972}). The differential algebraic subgroups of $\mathrm{SL}_2(F)$ were classified in \cite{sit:1975}. When the $\mathrm{PPV}$-group $\mathrm{Gal}_\Delta(M/K)$ is $\Pi$-constant, it is proved in \cite[Prop.~3.9(1)]{cassidy-singer:2006} that \eqref{ppv-eq} is \emph{completely integrable} \cite[Defn.~3.8]{cassidy-singer:2006}, and that $M$ is a Picard-Vessiot-extension (or $\mathrm{PV}$-extension) of $K$ for \eqref{ppv-eq}, in the non-parameterized sense of \cite{kolchin:1948}. The algorithms developed in \cite{gor-ov:2012} drastically reduce the number of conditions that one has to check in order to decide whether \eqref{ppv-eq} is completely integrable.

There is a \emph{parameterized Galois correspondence} between the differential algebraic subgroups $\Gamma$ of $\mathrm{Gal}_\Delta(M/K)$ and the intermediate $\Delta$-fields $K\subseteq L\subseteq M$, given by $\Gamma\mapsto M^\Gamma$ and $L\mapsto \mathrm{Gal}_\Delta(M/L)$. Under this correspondence, an intermediate $\Delta$-field $L$ is a $\mathrm{PPV}$-extension of $K$ (for some linear differential equation with respect to $\delta_x$) if and only if $\mathrm{Gal}_\Delta(M/L)$ is normal in $\mathrm{Gal}_\Delta(M/K)$; in which case the homomorphism $\mathrm{Gal}_\Delta(M/K)\twoheadrightarrow\mathrm{Gal}_\Delta(L/K)$, defined by $\sigma\mapsto\sigma|_L$, is surjective with kernel $\mathrm{Gal}_\Delta(M/L)$. See \cite[Thm. 3.5]{cassidy-singer:2006} and \cite[\S8.1]{gill-gor-ov:2012} for more details.


\section{Main result} \label{unipotent} 

We let $K:=F(x)$ denote the $\Delta$-field defined in \S\ref{ppv-intro}: $F=K^{\delta_x}$ is $\Pi$-closed field, $\delta_xx=1$, $\partial x=0$ for each $\partial\in\Pi$, and $\Delta:=\{\delta_x\}\cup\Pi$. Consider a second-order parameterized linear differential equation \begin{equation}\label{unimodular} \delta_x^2Y-qY=0,\end{equation} where $q\in K$. Let $M$ be a $\mathrm{PPV}$-field of $K$ for \eqref{unimodular}, and denote by $H:=\mathrm{Gal}_\Delta(M/K)$ the corresponding $\mathrm{PPV}$-group.


\subsection{Dreyfus' algorithm} \label{dreyfus} In \cite{dreyfus:2011} (see also \cite{arreche:2012}), Dreyfus develops a procedure to compute the $\mathrm{PPV}$-group $H$ corresponding to \eqref{unimodular}. We begin with a brief summary of this procedure in the non-reductive case (see \cite{minchenko-ovchinnikov-singer:2013a,minchenko-ovchinnikov-singer:2013b}).

When $H$ is not reductive \cite[Defn.~2.2.6]{minchenko-ovchinnikov-singer:2013b}, it is proved in \cite{kovacic:1986} that there exists an $F$-basis $\{\eta,\xi\}$ for the solution space $\mathcal{S}$ of \eqref{unimodular} such that $\delta_x\eta=u\eta$ for some $u\in K$, and $\label{relation2}\delta_x\bigl(\tfrac{\xi}{\eta}\bigr)=\eta^{-2}$. The embedding $H\hookrightarrow\mathrm{SL}_2(F)$ is defined by \begin{align*}\sigma(\eta) &=a_\sigma\eta \\ \sigma(\xi) & =a_\sigma^{-1}\xi+b_\sigma\eta,\end{align*} and there exist differential algebraic subgroups $A\leq \mathbb{G}_m(F)$ and $B\leq \mathbb{G}_a(F)$ such that the image of $H\hookrightarrow \mathrm{SL}_2(F)$ determined by this choice of basis \begin{equation}\label{borel} \left\{\begin{pmatrix} a_\sigma & b_\sigma \\ 0 & a_\sigma^{-1}\end{pmatrix} \ \middle| \ \sigma\in H\right\}= \left\{\begin{pmatrix} a & b \\ 0 & a^{-1}\end{pmatrix} \ \middle| \ a\in A, \ b\in B\right\}.\end{equation} If we let $L$ denote the $\mathrm{PPV}$-field corresponding to \begin{equation}\label{reductive}\delta_xY-uY=0,\end{equation} then the unipotent radical (see [Defn.~2.2.5]\cite{minchenko-ovchinnikov-singer:2013b}) $B\simeq R_u(H)$ coincides with the $\mathrm{PPV}$-group $\mathrm{Gal}_\Delta(M/L)$, and the maximal reductive quotient $A\simeq H/R_u(H)$ is naturally isomorphic to $\mathrm{Gal}_\Delta(L/K)$.

We refer to \cite{dreyfus:2011} for the computation of $A$ (see also \cite{kovacic:1986,singer:2011, chen-kauers-singer:2012, minchenko-ovchinnikov-singer:2013b, chen-etal:2013}). By the classification result of \cite[Prop.~11]{cassidy:1972}, $B$ is completely described by a finite set of linear differential operators $\mathbf{p}_1,\dots,\mathbf{p}_s\in F[\Pi]$, and it is shown in \cite[\S2]{dreyfus:2011} that they satisfy $\mathbf{p}_i(\eta^{-2})\in \delta_x(L)$ for $1\leq i\leq s$. When $A\subseteq\mathbb{G}_m(F^\Pi)$, it is proved in \cite{minchenko-ovchinnikov-singer:2013a} that, since $M$ is of finite algebraic transcendence degree over $L$, there are bounds on the orders of the $\mathbf{p}_i$. But if $A\nsubseteq \mathbb{G}_m(F^\Pi)$ and $B\neq 0$, then $M$ is of infinite algebraic transcendence degree over $L$ (see \cite{arreche:2013,minchenko-ovchinnikov-singer:2013a}). We are not aware of \emph{a priori} bounds on the orders of the $\mathbf{p}_i$ in this case (cf. \cite[p. 13]{minchenko-ovchinnikov-singer:2013a}), which raises the problem of deciding whether all the $\mathbf{p}_i$ have already been found, or whether it is still necessary to do more prolongations (see \cite{ovchinnikov:2009,minchenko-ovchinnikov-singer:2013a,minchenko-ovchinnikov-singer:2013b}).

In the setting of one parametric derivation $\Pi=\{\partial\}$, the problem of computing $R_u(H)$ was solved completely in \cite{arreche:2013}. When $H/R_u(H)$ is not $\partial$-constant, it is proved in \cite[Lem. 3.6(2)]{hardouin-singer:2008} that either $R_u(H)=\mathbb{G}_a(F)$ or $R_u(H)=0$, and there are no other possibilities. This has the consequence, counterintuitive in light of \cite{minchenko-ovchinnikov-singer:2013a}, that the computation of $R_u(H)$ is actually easier when $H/R_u(H)$ is \emph{not} $\partial$-constant (see \cite[Thm. 3.2]{arreche:2013}), since in this case the parametric derivation $\partial$ is barred from appearing in the defining equations for $R_u(H)$. Theorem~\ref{main1} describes the generalization of this phenomenon to the the setting of several parametric derivations.


\subsection{Computation of the unipotent radical}

Let $\mathcal{D}:=F\cdot\Pi$, the $F$-linear span of $\Pi$. Consider the set $\mathcal{L}$ of derivations $\partial\in \mathcal{D}$ such that every element of $A\simeq H/R_u(H)$ is constant with respect to $\partial$: \begin{equation} \label{lie}\mathcal{L}:=\{\partial\in\mathcal{D} \ | \ \partial a=0, \ \forall a \in A\}.\end{equation} A computation shows that $\mathcal{L}$ is a \emph{Lie subspace} of $\mathcal{D}$, i.e., an $F$-subspace that is closed under the Lie bracket on derivations. By \cite[Prop. 39]{cassidy:1972} and \cite[Prop. 0.6]{kolchin:1985}, there exists a commuting $F$-basis $\Pi':=\{\partial_1',\dots,\partial_k'\}$ for $\mathcal{L}$.

Now let $\Delta':=\{\delta_x\}\cup\Pi'$, and consider $K$ as a $\Delta'$-field. Then, the $\Delta'$-field $M':=K\langle\eta,\xi\rangle_{\Delta'}$ is a $\mathrm{PPV}$-extension of $K$ for \eqref{unimodular}, and a $\Delta'$-subfield of $M$. We identify the $\mathrm{PPV}$-group $H':=\mathrm{Gal}_{\Delta'}(M'/K)$ with a $\Pi'$-subgroup of $\mathrm{SL}_2(F)$ by means of the same basis $\{\eta,\xi\}$, and define $A'$ and $B'$ as in \eqref{borel}.


\begin{rem} \label{piprime} Let us briefly describe how to compute $\Pi'$. Observe that, for every $\partial\in\mathcal{D}$ and $\sigma\in H$, \begin{equation}\label{piprime1}\sigma\bigl(\tfrac{\partial\eta}{\eta}\bigr)=\tfrac{\partial\eta}{\eta}+\tfrac{\partial a_\sigma}{a_\sigma}\qquad\quad\text{and}\qquad\quad\delta_x\bigl(\tfrac{\partial\eta}{\eta}\bigr)=\partial u.\end{equation} Hence, the parameterized Galois correspondence implies that \begin{equation}\label{piprime2}\mathcal{L}=\bigl\{\partial\in\mathcal{D} \ \big| \ \tfrac{\partial\eta}{\eta}\in K \bigr\}=\{\partial\in\mathcal{D} \ | \ \partial u\in\delta_x(K)\}.\end{equation} By writing a derivation $\partial=\sum_i c_i\partial_i$ with undetermined coefficients, and applying Hermite reduction to $\partial u$ (see \cite{chen-kauers-singer:2012, chen-etal:2013}), the condition $\partial u\in\delta_x(K)$ becomes an $F$-linear condition on the coefficients $c_i$. Thus the computation of a (possibly non-commuting) basis $\Pi''$ for $\mathcal{L}$ is reduced to linear algebra. The proof of \cite[Prop.~0.6]{kolchin:1985} gives an algebraic recipe to produce a commuting basis $\Pi'$ for $\mathcal{L}$ from the (possibly non-commuting) basis $\Pi''$. This recipe was generalized and applied towards an algorithm to decide isomonodromy in \cite{gor-ov:2012}. Corollary~\ref{emptyprime} below, together with \cite[Lem.~4.3]{arreche:2013}, gives a simple and effective test to decide whether or not $R_u(H)\simeq\mathbb{G}_a(F)$.\end{rem}


\begin{thm}[Main result] \label{main1} The reductive quotient $H'/R_u(H')$ is $\Pi'$-constant, and the defining operators $\{\mathbf{p}_i\}_{i=1}^s \subset F[\Pi']$ for $R_u(H')$ are also defining operators for $R_u(H)$, under the natural inclusion $F[\Pi']\subseteq F[\Pi]$.
\end{thm}

\begin{proof}That $A'$ is $\Pi'$-constant follows from Remark~\ref{piprime}: since $\frac{\partial \eta}{\eta}\in K$ for each $\partial\in\Pi'$, we have that $\frac{\partial a}{a}=0$ for each $a\in A'$. We will prove that $B=B'$ in a series of lemmas. By Lemma~\ref{inj-res}, we have that $B\subseteq B'$. By Lemma~\ref{surj-res1}, there is a finite set $\{\mathbf{p}_i\}_{i=1}^s\subset F[\Pi']$ such that $B$ coincides with the set of those $b\in \mathbb{G}_a(F)$ such that $\mathbf{p}_i(b)=0$ for each $1\leq i \leq s$. By Lemma~\ref{surj-res2}, $\mathbf{p}_i(b')=0$ for each $b'\in B'$ and $1\leq i \leq s$, whence $B'\subseteq B$. \end{proof}

\begin{cor}[cf. \protect{\cite[Lem.~3.6(2)]{hardouin-singer:2008}}]\label{emptyprime} Suppose that $R_u(H)\neq \{ 0 \}$. Then, $$R_u(H)\simeq\mathbb{G}_a(H) \quad\qquad\Longleftrightarrow \quad\qquad\mathcal{L}=\{0\}.$$\end{cor}

\begin{proof} If $\mathcal{L}=\{ 0\}$, then $H'$ is the $\mathrm{PV}$-group corresponding to \eqref{unimodular}, and $B'$ is an algebraic subgroup of $\mathbb{G}_a(F)$. By Lemma~\ref{inj-res}, $B'\neq\{ 0\}$, because $B\subseteq B'$ and $B\neq\{0\}$. Therefore, $B'=\mathbb{G}_a(F)=B$, by Theorem~\ref{main1}.

Supposing instead that $\mathcal{L}\neq \{0 \}$, we have that $\Pi'\neq\emptyset$. If $A'$ is finite, the fact that $B'\neq \mathbb{G}_a(F)$ follows from \cite[Prop.~3.3]{singer:2011}. If $A'$ is infinite, since $A'$ is $\Pi'$-constant by Theorem~\ref{main1}, the classification of \cite[\S IV.1]{cassidy:1972} says that $A'=\mathbb{G}_m(F^{\Pi'})$. That $B'\neq\mathbb{G}_a(F)$ follows from \cite[pp.~159--160]{singer:2011} in this case (see also \cite[proof of Prop.~4.4]{arreche:2013} and \cite[Alg.~1]{minchenko-ovchinnikov-singer:2013a}). By Theorem~\ref{main1}, $B\neq \mathbb{G}_a(F)$. \end{proof}

The following three lemmas were used in the proof of Theorem~\ref{main1}.


\begin{lem} \label{inj-res} The restriction homomorphism $H\hookrightarrow H':\sigma\mapsto\sigma|_{M'}$ induces an inclusion $R_u(H)\hookrightarrow R_u(H')$. \end{lem}

\begin{proof} The actions of $H$ and $H'$ on $M$ and $M'$ are completely determined by their restrictions to the same solution space $\mathcal{S}=F\cdot\eta\oplus F\cdot\xi$, whose definition is independent of the chosen set of parametric derivations. Hence, the restriction homomorphism $H\hookrightarrow H'$ is injective, and it is clear from the definitions that $R_u(H)$ is then mapped (injectively) into $R_u(H')$. \end{proof}


The fact that $B$ is the unipotent radical of \eqref{borel}, and not just any differential algebraic subgroup of $\mathbb{G}_a(F)$, allows us to produce a set of defining operators for $B$ from $F[\Pi']\subseteq F[\Pi]$, which sharpens the classification result of \cite[Prop.~11]{cassidy:1972} in this very particular case. The following structural result, which was inspired by the results of \cite{sit:1975} cited in its proof, holds true for any linear differential algebraic group $G$ of the form \eqref{borel}, whether or not it happens to be a $\mathrm{PPV}$-group over $K$.

\begin{lem}[cf. \protect{\cite[Lem. 3.6(2)]{hardouin-singer:2008}, \cite[Thm. II.1.3 and Thm. II.1.4]{sit:1975}}] \label{surj-res1} There exist finitely many linear differential operators $\mathbf{p}_1,\dots,\mathbf{p}_s\in F[\Pi']\subseteq F[\Pi]$ such that $$B=\{b\in F \ | \ \mathbf{p}_i(b)=0, \ 1\leq i \leq s\}.$$\end{lem} 

\begin{proof}By \cite[Prop. 0.7]{kolchin:1985} the $F$-basis $\Pi'$ for $\mathcal{L}$ can be extended to a commuting $F$-basis $\tilde{\Pi}:=\{\partial_1',\dots,\partial_m'\}$ for all of $\mathcal{D}$. We denote by $\tilde{\Theta}$ (resp., $\Theta'$) the free commutative monoid generated by $\tilde{\Pi}$ (resp., $\Pi'$). Consider the orderly ranking on $F\{Y\}_{ \tilde{\Pi}}$ determined by the lexicographic order on $\tilde{\Theta}$ defined by setting $\delta_i'\leq\delta_j'$ if $i\leq j$. In other words, to compare two elements $\theta,\theta'$ in $\tilde{\Theta}$, we first compare their total orders, and then the exponents of $\partial'_m,\dots,\partial_1'$, in that order.

By \cite[Thm. II.1.3(b) and Thm. II.1.4]{sit:1975}, there is a characteristic set $\{\mathbf{p}_1,\dots,\mathbf{p}_s\}$ for the defining ideal of $B$ (with respect to this ranking) such that $\mathbf{p}_i(aY)=a\mathbf{p}_i(Y)$ for each $a\in A$ and $1\leq i\leq s$. Therefore, to show that $\{\mathbf{p}_i\}_{i=1}^s\subset F[\Pi']$, it suffices to prove that if $\mathbf{p}\in F[\tilde{\Pi}]$ does not belong to the image of $F[\Pi']$ under the natural embedding $F[\Pi']\subseteq F[\Pi]$, then there exists an element $a\in A$ such that $\mathbf{p}(aY)-a\mathbf{p}(Y)\neq 0$. 

So suppose that $\mathbf{p}\in F[\tilde{\Pi}]$ and $\mathbf{p}\notin F[\Pi']$, and let $c_\theta\theta Y$ be the monomial in $\mathbf{p}$ of highest rank such that $c_\theta\neq 0$ and $\theta$ contains a derivation $$\partial'\in\tilde{\Pi}\backslash\Pi'=\{\partial_{k+1}',\dots,\partial_m'\}.$$ Assume that $\partial_\ell'$ is the derivation of highest rank appearing effectively in $\theta$, and let $\tilde{\theta}$ denote the element of $\tilde{\Theta}$ obtained from $\theta$ by decreasing the order of $\partial_\ell'$ by $1$. Since $\theta'(aY)=a\theta'Y$ for every $a\in A$ and $\theta'\in \Theta'$, the leader of $\mathbf{p}(aY)-a\mathbf{p}(Y)$ is $c_\theta\partial_\ell'(a)\tilde{\theta}Y$ whenever $a\in A$. Since $\partial_\ell'\notin\mathcal{L}$, there exists an element $a\in A$ such that $\partial_\ell'(a)\neq 0$, whence $\mathbf{p}(aY)-a\mathbf{p}(Y)\neq 0$. \end{proof}

\begin{rem} \label{surj1-rem} When $A$ is $\Pi$-constant, we may take $\Pi'=\Pi$, and Lemma~\ref{surj-res1} coincides with \cite[Prop.~11]{cassidy:1972}. In the case that $\Pi=\{\partial\}$ is a singleton, Lemma~\ref{surj-res1} specializes to \cite[Lem.~3.6(2)]{hardouin-singer:2008}.
\end{rem}


The previous result shows that $B$ can be defined as a subset of $\mathbb{G}_a(F)$ using derivations from $\Pi'$ only. The following result rules out the possibility that $B$ could somehow be defined by more $\Pi'$-differential equations than $B'$.

\begin{lem}[cf. \protect{\cite[Prop. 3.6(2)]{cassidy-singer:2006}}] \label{surj-res2} If $\mathbf{p} \in F[\Pi']$ is such that $\mathbf{p}(b)=0$ for every $b\in B$, then $\mathbf{p}(b')=0$ for every $b'\in B'$. In other words, $B\subseteq B'$ is $\Pi'$-dense.
\end{lem}

\begin{proof} Suppose that $\mathbf{p}\in F[\Pi']$ is such that $\mathbf{p}(b)=0$ for each $b\in B$. Then by \cite[\S2.1, p. 7]{dreyfus:2011}, we have $\mathbf{p}(\eta^{-2})\in \delta_x(L)$. Moreover, since $\mathbf{p}\in F[\Pi']$, $$\mathbf{p}(\eta^{-2})\in K\langle\eta\rangle_{\Delta'}=:L',$$ the fixed field of $R_u(H')$. We will show that in fact $\mathbf{p}(\eta^{-2})\in\delta_x(L')$. Again by \cite[\S2.1, p. 7]{dreyfus:2011}, this will imply that $\mathbf{p}(b')=0$ for each $b'\in B'$. Assume that $\eta$ is algebraically transcendental over $K$, since otherwise $A\simeq\mu_k$, the group of $k^\text{th}$ roots of unity (see \cite[Prop.~31]{cassidy:1972}), whence $\Pi=\Pi'$ and there is nothing to show.
 
By \cite[Prop. 3.9]{cassidy-singer:2006} (cf. Remark~\ref{piprime}), the fact that $A'$ is $\Pi'$-constant implies that \begin{equation}\label{aprimeconst}v_j:=\tfrac{\partial_j'\eta}{\eta}\in K\end{equation} for each $\partial_j'\in\Pi'$, and therefore $L'=K(\eta)$. It also follows from \eqref{aprimeconst} that \begin{equation}\label{veejay}-2v_j=\eta^2\partial_j'(\eta^{-2})\in K.\end{equation} Let us prove by induction that $\eta^2\theta'(\eta^{-2})\in K$ for each $\theta'\in\Theta'$, the free commutative monoid on $\Pi'$. The base case is \eqref{veejay}. Assuming that $\eta^2\theta'(\eta^{-2})=:v_{\theta'}\in K$, then $$\eta^2\partial_j'\theta'(\eta^{-2})=\eta^2\partial_j'(v_{\theta'}\eta^{-2})=\partial_j'v_{\theta'}-2v_jv_{\theta'}\in K$$ proves the induction step, and our claim. Hence, $\eta^2\mathbf{p}(\eta^{-2})\in K$ for every $\mathbf{p}\in F[\Pi']$. 

Since $$L:=K\langle\eta\rangle_\Delta=K(\eta)\langle\partial_1\eta,\dots,\partial_m\eta\rangle_\Delta=K(\eta)\bigl\langle\tfrac{\partial_1\eta}{\eta},\dots,\tfrac{\partial_m\eta}{\eta}\bigr\rangle_\Delta,$$ it follows that $L$ is algebraically generated as a field extension of $L'=K(\eta)$ by \begin{equation}\label{generators}\bigl\{\theta\tfrac{\partial_j\eta}{\eta} \ | \ \theta\in\Theta, \ 1\leq j \leq m\bigr\},\end{equation} where $\Theta$ is the free commutative monoid on $\Pi$. By \cite[Cor.~5.1.2]{minchenko-ovchinnikov-singer:2013b} and \cite[Prop.~3.2]{minchenko-ovchinnikov-singer:2013a} (see also \cite[\S3.2.1]{minchenko-ovchinnikov-singer:2013a}), if we consider $L$ and $K$ as $\delta_x$-fields, then $L$ is a (non-parameterized) $\mathrm{PV}$-extension of $K$, and the algebraic transcendence degree of $L$ over $K$ is finite. Hence, we may choose a finite set $\beta_1,\dots,\beta_S$ of $F$-linearly independent generators for $L$ over $L'$ from the set \eqref{generators}. It follows from~\eqref{piprime1} that $\delta_x\beta_i\in K$ for each $1\leq i \leq S$. By the Kolchin-Ostrowski theorem \cite{kolchin:1968}, the $\beta_i$ are then algebraically independent over $L'$. We define $$ N:=K(\beta_1,\dots,\beta_S),$$ and observe that $L= N(\eta)$. Since $A$ is abelian, the subgroup $\mathrm{Gal}_\Delta(L/ N)\leq A$ is normal, and therefore $ N$ is a $\mathrm{PPV}$-extension of $K$ by the parameterized Galois correspondence \cite[Thm.~3.5]{cassidy-singer:2006}. Again by \cite[Cor.~5.1.2]{minchenko-ovchinnikov-singer:2013b} and \cite[Prop.~3.2]{minchenko-ovchinnikov-singer:2013a}, the $\delta_x$-field $ N$ is a (non-parameterized) $\mathrm{PV}$-extension of the $\delta_x$-field $K$. Since $$\tfrac{\delta_x\eta}{\eta}=u\in K\quad\qquad \text{and}\quad\qquad \delta_x\beta_i\in K$$ for each $1\leq i \leq S$, the Kolchin-Ostrowski theorem \cite{kolchin:1968} implies that $\eta$ is algebraically transcendental over $ N$. The corresponding $\mathrm{PV}$-ring is \begin{equation}\label{pv-ring}P:=K[\beta_1,\dots,\beta_S]\subset  N.\end{equation}

Let $f\in L$ such that $\delta_x(f)=\mathbf{p}(\eta^{-2})$. We claim that there exist elements $g\in N$ and $c\in F$ such that $f=g\eta^{-2}+c$. To see this, let $h\in K$ be such that $\mathbf{p}(\eta^{-2})=h\eta^{-2}$, and write the partial fraction decomposition of $f$ considered as a rational function in $\eta$, where the coefficients $c_i$, $e_k$, $g_{j,k}\in \bar{N}$ belong to some algebraic closure $\bar{N}$ of $ N$: \begin{equation}\label{parfrac1} \sum_i c_i\eta^i+\sum_{j,k}\frac{g_{j,k}}{(\eta-e_k)^j}=f.\end{equation} Let $e_0=0$, and apply $\delta_x$ on both sides of \eqref{parfrac1} to obtain (cf. \cite[Lem.~2.1]{arreche:2012}): \begin{equation}\label{parfrac2} \sum_i (\delta_xc_i +iuc_i)\eta^i+\sum_{j,k} \frac{\delta_xg_{j,k}-jug_{j,k}}{(\eta-e_k)^j}+\frac{jg_{j,k}(\delta_xe_k-ue_k)}{(\eta-e_k)^{j+1}}  =\delta_xf=\mathbf{p}(\eta^{-2})=h\eta^{-2}.\end{equation} Comparing coefficients in \eqref{parfrac2} shows that $\delta_xc_0=0$ and that $\delta_xc_i=-iuc_i$, which implies that either $c_i=a\eta^i$ for some $a\in F$, or else $c_i=0$. In any case $c_0\in F$, snd since $\eta$ is algebraically transcendental over $ N$, $c_i=0$ for $i>0$. Now fix $k>0$, so that $e_k\neq 0$, and let $j>0$ be the smallest integer such that $g_{j,k}\neq 0$. Again comparing coefficients in \eqref{parfrac2}, we obtain that $\delta_xg_{j,k}=jug_{j,k}$, which implies that $g_{j,k}=a\eta^j$ for some $0\neq a\in F$. This is impossible, and therefore there is no such $j$, and only $k=0$ appears in the sum \eqref{parfrac1}. Finally, since $e_0=0$, we obtain that $\delta_xg_{j,0}=jug_{j,0}$ for $j\neq 2$ by comparing coefficients in \eqref{parfrac2}, which again implies that $g_{j,0}=0$ whenever $j\neq 2$. Therefore, $$f=g_{2,0}\eta^{-2}+c_0,$$ where $c_0\in F$ and $g_{2,0}\in\bar{N}$ is \emph{algebraic} over $ N$. Since $$g_{2,0}=\eta^2(f-c_0)\in L= N(\eta),$$ the fact that $\eta$ is algebraically transcendental over $ N$ implies that $g_{2,0}\in N$.

Having shown that $f=g\eta^{-2}+c$ for some $g\in N$ and $c\in F$, let us now show that the element $g$ actually belongs to $K$, which implies that $f\in L'$. Since $$\bigl(\delta_xg-2ug\bigr)\eta^{-2}=\delta_xf=\mathbf{p}(\eta^{-2})=h \eta^{-2}$$ for some $h\in K$, it follows that \begin{equation}\label{beta-rel}\delta_xg-2ug=h.\end{equation} We begin by showing that $g\in K[\beta_1,\dots,\beta_S]$ must be a polynomial expression in the $\beta_i$. Indeed, it follows from \eqref{beta-rel} that the $K$-vector space $\sum_jK\cdot\delta_x^jg\subset  N$ is finite-dimensional over $K$. By \cite[Cor. 1.38]{vanderput-singer:2003}, the finite-dimensionality of $\sum_jK\cdot\delta_x^jg$ over $K$ is a necessary and sufficient condition for $g\in N$ to belong to the $\mathrm{PV}$-ring $P$ defined in \eqref{pv-ring}. 

To show that $g\in K$, we proceed by contradiction. Suppose that $r_I\underline{\beta}^I$ is a monomial in $g$, considered as a polynomial in the $\beta_i$, with $0\neq |I|$ maximal and $0\neq r_I\in K$. Since the coefficient of $\underline{\beta}^I$ in the right-hand side of \eqref{beta-rel} is $0$, we see that $\delta_xr_I=2ur_I$, which implies that $r_I=a\eta^2$ for some $a\in F^\times$. Since $\eta^2\notin K$, no such monomial $r_I\underline{\beta}^I$ appears in $g$, which means that $g\in K$. Therefore, $g\eta^{-2}+c=f\in L'$, which concludes the proof of the Lemma. \end{proof}

\begin{rem} \label{surj2-rem} In case $\Pi'=\emptyset$, or in other words when the Lie subspace $\mathcal{L}$ defined in \eqref{lie} is $\{0\}$, then $H'$ is the (non-parameterized) $\mathrm{PV}$-group for \eqref{unimodular}, and Lemma~\ref{surj-res2} reduces to a special case of \cite[Prop.~3.6(2)]{cassidy-singer:2006}.
\end{rem}


\section{An example}\label{examples}

We let $K=F(x)$ denote the $\Delta$-field defined in the previous section, where $\Pi:=\{\partial_1,\partial_2\}$, $\partial_j:=\tfrac{\partial}{\partial t_j}$ for $j=1,2$, and $F$ denotes a $\Pi$-closed field containing $\mathbb{Q}(t_1,t_2)$ (see \cite{kolchin:1974,trushin:2010}). In this section, we will apply Theorem~\ref{main1} to compute the $\mathrm{PPV}$-group $H$ corresponding to the parameterized linear differential equation \begin{equation}\label{eg1} \delta_x^2Y-\left(\frac{x^2+(2-2t_1t_2)x+t_1^2t_2^2-3t_1t_2+2}{x^2}\right)Y=0.\end{equation} The Riccati equation \begin{gather*} \delta_xu+u^2=\tfrac{x^2+(2-2t_1t_2)x+t_1^2t_2^2-3t_1t_2+2}{x^2}=:q \intertext{admits the solution} u=\frac{t_1t_2-1-x}{x}. \end{gather*} Therefore, by \cite{kovacic:1986} there is a basis $\{\eta,\xi\}$ for the solution space of \eqref{eg1} such that $\delta_x\eta=u\eta$ and $\delta_x\bigl(\frac{\xi}{\eta}\bigr)=\eta^{-2}$, and by \cite[\S2.1]{dreyfus:2011} there exist differential algebraic subgroups $A\leq\mathbb{G}_m(F)$ and $B\leq \mathbb{G}_a(G)$ such that $H$ is given by \eqref{borel}. Since $\partial_1u=\tfrac{t_2}{x}$ and $\partial_2u=\tfrac{t_1}{x},$ we have that \begin{equation}\label{ynot} \partial_1^2u=0=\partial_2^2u \quad\qquad\text{and}\quad\qquad t_1\partial_1u-t_2\partial_2u=0.\end{equation} Hence the $F[\Pi]$-span of $\{\partial_1u,\partial_2u\}$ and the $F$-span of $\{\partial_1u,\partial_2u,\partial_1\partial_2u\}$ are the same modulo $\delta_x(K)$. By \cite[\S2.1]{dreyfus:2011}, $$A=\bigl\{a\in\mathbb{G}_m(F) \ \big| \ t_1\tfrac{\partial_1a}{a}=t_2\tfrac{\partial_2a}{a}; \ \partial_1\bigl(\tfrac{\partial_1a}{a}\bigr)=0=\partial_2\bigl(\tfrac{\partial_2a}{a}\bigr)\bigr\}.$$

Since $\partial_1u\notin\delta_x(K)$, the Lie subspace of derivations $\mathcal{L}\subset F\cdot \Pi$ defined in \eqref{lie}, or equivalently in \eqref{piprime2}, has dimension at most $1$ over $F$. Hence, by \eqref{ynot} $\mathcal{L}$ coincides with $F\cdot (t_1\partial_1-t_2\partial_2)$, the $F$-span of $\partial_1':=t_1\partial_1-t_2\partial_2$. Hence, we may take $\Pi':=\{\partial_1'\}$.

By Theorem~\ref{main1}, to compute the unipotent radical $R_u(H)=B$ it suffices to compute the unipotent radical $R_u(H')=:B'$, where $H'$ denotes the $\mathrm{PPV}$-group of \eqref{eg1} relative to the new set of parametric derivations $\Pi'=\{\partial_1'\}$. It follows from \eqref{ynot} that the system \begin{equation}\label{eg-pv} \begin{cases} \delta_xY-uY=0 \\ \partial_1'Y=0\end{cases}\end{equation} is \emph{isomonodromic} \cite{gor-ov:2012} (or \emph{completely integrable}, in the terminology of \cite[Defn.~3.8]{cassidy-singer:2006}). Therefore, by \cite[Prop.~3.9]{cassidy-singer:2006} $L'=K(\eta)$ is a (non-parameterized) $\mathrm{PV}$-extension of $K$ for \eqref{eg-pv}, and (cf. Theorem~\ref{main1}) $$\mathrm{Gal}_{\Delta'}(L'/K)\simeq H'/R_u(H')\simeq A'=\mathbb{G}_m(F^{\partial_1'}).$$ 

By \cite[Lem.~4.3]{arreche:2013} and \cite[Prop.~2.6(2)]{cassidy-singer:2006}, to see that $B'\neq 0$, it suffices to show that there is no $f\in K$ such that $\delta_xf+2uf=1$. We prove this by contradiction, along the lines of \cite[proof of Cor.~3.3]{arreche:2013}. Assume that $f\in K$ satisfies \begin{equation}\label{bnot0} \delta_xf+2uf=1.\end{equation} First, note that $f$ cannot be $\delta_x$-constant, whence it must a have a pole somewhere in $\mathbb{P}^1(F)$. But $f$ cannot have a pole outside of $\{0,\infty\}$, for otherwise the left-hand side of \eqref{bnot0} would have a pole. If $f$ had a pole at $0$, the residue of $2u$ at $0$ would have to be an integer, which is clearly false. Therefore, $f$ can only have a pole at $\infty$, i.e., $f$ is a polynomial in $x$. Moreover, $f$ must be divisible by $x$, because otherwise the left-hand side of \eqref{bnot0} would have a pole at $0$. But then the degree of the polynomial on the left-hand side of \eqref{bnot0} is equal to the degree of $f$, which is at least $1$, since $f$ is not constant. This contradiction concludes the proof that there is no solution in $K$ for \eqref{bnot0}, and therefore that $B'\neq 0$. Since $$\delta_x\bigl(\partial_1'\tfrac{\xi}{\eta}\bigr)=\partial_1'\eta^{-2}=0 \qquad\Longrightarrow \qquad\partial_1'\tfrac{\xi}{\eta}\in F=\delta_x(F\cdot x)\subset \delta_x(K),$$  it follows from \cite[\S2.1, p.~7]{dreyfus:2011} that $B'=\mathbb{G}_a(F^{\partial_1'})$. Therefore, by Theorem~\ref{main1}, $$H\simeq\left\{\begin{pmatrix} a & b \\ 0 & a^{-1} \end{pmatrix} \ \middle| \ \begin{matrix}a,b \in F; \quad a\neq 0; \quad t_1\tfrac{\partial_1a}{a}=t_2\tfrac{\partial_2a}{a}; \\ \\ \partial_1\bigl(\tfrac{\partial_1a}{a}\bigr)=0=\partial_2\bigl(\tfrac{\partial_2a}{a}\bigr); \quad t_1\partial_1b=t_2\partial_2b\end{matrix}\right\}.$$

\bibliographystyle{spmpsci}

 \end{document}